\documentclass[10pt]{amsart}
\usepackage[english]{babel}
\usepackage[T1]{fontenc}
\usepackage[latin1]{inputenc}
\usepackage{mathptmx}
\usepackage{amsmath,amssymb}
\usepackage{amsthm,amsfonts}
\usepackage{dsfont}
\usepackage{color}
\usepackage{enumitem}
\newcommand{\CC}{\mathbb{C}}

\newcommand{\NN}{\mathbb{N}}
\newcommand{\Zn}{\mathbb{Z}_n}
\newcommand{\KP}{\mathbb{KP}}
\newcommand{\Tr}{\mathrm{Tr}}
\newcommand{\cop}{\Delta}

\newcommand{\hw}{\int_{\KP}\!\!}
\newcommand{\X}{\mathfrak{X}}
\newcommand{\hwn}{\int_{\KP_n}\!\!}

\newcommand{\indic}{\mathds{1}}

\newcommand{\Xuv}[2]{\X^{#1,#2}}
\newcommand{\modu}[1]{\left\lvert #1 \right\rvert}
\newcommand{\norm}[1]{\left\lVert #1 \right\rVert}
\newcommand{\Fo}{\mathcal{F}}
\DeclareMathOperator{\lcm}{lcm}
\DeclareMathOperator{\id}{id}
\DeclareMathOperator{\LinSpan}{span}

\newtheorem{theorem}{Theorem}[section]

\newtheorem{lemma}{Lemma}[section]
\newtheorem{proposition}{Proposition}[section]
\newtheorem{definition}{Definition}[section]
\newtheorem{example}{Example}[section]
\newtheorem{remark}{Remark}[section]

\title{Random walks on finite quantum groups}
\author{Isabelle Baraquin}
\address{Laboratoire de math\'ematiques de Besan\c{c}on, UMR CNRS 6623, Universit\'e Bourgogne Franche-Comt\'e, 16 route de Gray, 25030 Besan\c{c}on cedex, France}
\email{isabelle.baraquin@univ-fcomte.fr}

\begin{document}


\begin{abstract}
In this paper we study convergence of random walks, on finite quantum groups, arising from linear combination of irreducible characters. We bound the distance to the Haar state and determine the asymptotic behavior, i.e. the limit state if it exists. We note that the possible limits are any central idempotent state. We also look at cut-off phenomenon in the Sekine finite quantum groups.
\end{abstract}

\maketitle

\section{Introduction}

Given a finite group $G$, and a current position $h$ in $G$, choose randomly, according to a fixed probability distribution, an element $g$, and move to $gh$. If you repeat this procedure, with the same probability distribution $\pi$, you will get random positions $h_k = g_k g_{k-1} \ldots g_{1}$ which define a walk on the Cayley graph of $G$. A frequently asked question is when the distribution $\pi^{\star k}$ of the current position $h_k$ is close to the uniform distribution on $G$. Diaconis and Shahshahani answer this question in \cite{RwSn} for the case of permutation groups, by using representation theory. They also observe that after less than $t_n$ random transpositions a deck of $n$ cards is not well mixed, but after more than $t_n$ random transpositions it is well mixed, where $t_n = \frac{1}{2}n \log(n)$. This is called a cut-off. For more details on the cut-off phenomenon in classical finite groups, the reader can look at the survey of Saloff-Coste \cite{SaloffCoste}.

Finite quantum groups are a generalization of classical finite groups. Despite of their name of group, we are studying algebras endowed with an additional structure.
\begin{definition}
A finite quantum group is a pair  $\mathbb{G} = (\mathcal{A}, \cop)$, where $\mathcal{A}$ is a unital finite dimensional $*$-algebra (eventually noncommutative) and $\cop \colon \mathcal{A} \to \mathcal{A}\otimes \mathcal{A}$ is a unital $*$-homomorphism, called coproduct, such that it is coassociative, i.e.
\[\left( \cop \otimes \id_{\mathcal{A}} \right) \circ \cop = \left( \id_{\mathcal{A}} \otimes \cop \right) \circ \cop\]
as $*$-homomorphisms from $\mathcal{A}$ to $\mathcal{A} \otimes \mathcal{A} \otimes \mathcal{A}$, and such that it verifies the density in $\mathcal{A} \otimes \mathcal{A}$ of the two algebras $(1_{\mathcal{A}} \otimes \mathcal{A}) \cop(\mathcal{A})  = \LinSpan\left\{(1_{\mathcal{A}} \otimes a)\cop(b) , a, b \in \mathcal{A}\right\}$ and $(\mathcal{A} \otimes 1_{\mathcal{A}})\cop(\mathcal{A}) = \LinSpan\left\{(a \otimes 1_{\mathcal{A}})\cop(b) , a, b \in \mathcal{A}\right\}$.
\end{definition}

Note that we can define a finite quantum group thanks to a classical finite group $(G,\cdot)$. The set of complex-valued continuous functions on $G$, denoted $C(G)$, is endowed with a structure of (commutative) unital $*$-algebra. Identifying the tensor product $C(G) \otimes C(G)$ with the $*$-algebra $C(G \times G)$, one can check that the application $\cop \colon C(G) \to C(G \times G)$, defined by $\cop(f)(s,t) = f(s\cdot t)$ satisfies the coassociativity relation and the density properties, called quantum cancellation rules. Following this example, we denote the algebra $\mathcal{A}$ of the quantum group $\mathbb{G} = (\mathcal{A}, \cop)$ by $C(\mathbb{G})$.

We need to transfer the properties of probability measures defined on a finite group $G$, to some object defined on $C(G)$. Let us note that if we look at integration of continuous functions with respect to a probability measure, we obtain a state on $C(G)$, i.e.\ a linear functional which preserves the unit and positivity. In the framework of quantum groups, the convolution product of two probability measures $\mu_1$ and $\mu_2$ is given by $\mu_1 \star \mu_2 = (\mu_1 \otimes \mu_2) \circ \cop$. This formula defines the convolution product of two states defined on a quantum group $\mathbb{G} = (C(\mathbb{G}), \cop)$.

Moreover, every classical finite group can be equipped with a Haar measure $\lambda$, satisfying the identity
\[\int_G f(g \cdot s)\,\mathrm{d}\lambda(s) = \int_G f(s)\,\mathrm{d}\lambda(s) = \int_G f(s \cdot g)\,\mathrm{d}\lambda(s)\]
which becomes $(\lambda \otimes \id_{C(G)}) \circ \cop = \lambda(\cdot) \indic_G = (\id_{C(G)} \otimes \lambda) \circ \cop$, with the coproduct defined above, where $\lambda$ denotes the integration with respect to the Haar measure. This leads us to the following definition:
\begin{definition}
A Haar state $\int_{\mathbb{G}}$ on a compact quantum group $\mathbb{G} = (C(\mathbb{G}), \cop)$ is a state on $C(\mathbb{G})$ such that
\[(\int_{\mathbb{G}} \otimes \id_{C(\mathbb{G})})\circ \cop = 1_{C(\mathbb{G})} \int_{\mathbb{G}}(\cdot) = (\id_{C(\mathbb{G})} \otimes \int_{\mathbb{G}})\circ \cop\text{ .}\]
\end{definition}
Note that every finite quantum group admits a unique Haar state.

Random walks on finite quantum groups were first studied by Franz and Gohm \cite{uweKP}. We are now looking at convolution semigroups of states $\{\phi^{\star k}\}_{k \in \NN}$ and their convergence to the Haar state. Recently, McCarthy \cite{McCarthy} developed the Diaconis-Shahshahani theory for finite quantum groups. In particular, he gave upper and lower bounds for the distance to the Haar state, exploiting the corresponding tools, but he could not prove a cut-off in the Sekine family of finite quantum groups. He defined  \cite{cutOffJP} the quantum total variation distance as a norm on functions defined on $C(\mathbb{G})$ by
\[\norm{\phi - \psi}_{QTV} = \sup_{p \in C(\mathbb{G}) \text{ a projection}} \modu{\phi(p) - \psi(p)} \text{ .}\]

For a finite quantum group $\mathbb{G}$, let us denote by $I(\mathbb{G})$ the set of its irreducible representations, i.e.\ invertible elements $\alpha = (\alpha_{ij})_{1 \leq i,j \leq n}$ of $M_n(C(\mathbb{G}))$ such that, for any $1 \leq i,j \leq n$, $\cop(\alpha_{ij}) = \sum\limits_k \alpha_{ik} \otimes \alpha_{kj}$
For a representation $\alpha \in I(\mathbb{G})$, $d_\alpha$ denotes its dimension and $\chi_\alpha$ the associated character.

We shall also use the Fourier transform. For $x$ an element in $C(\mathbb{G})$, let us denote by $\Fo(x)$ the linear functional defined on $C(\mathbb{G})$ by
\[\Fo(x)(y) = \int_{\mathbb{G}}\!\! yx \text{ .}\]
This allows us to define a convolution product on $C(\mathbb{G})$, also denoted $\star$ like the product of linear functionals on $C(\mathbb{G})$, thanks to the formula
\[\Fo(a \star b) = \Fo(a) \star \Fo(b) := \left(\Fo(a) \otimes \Fo(b)\right) \circ \Delta \text{ .}\]
Using Sweedler notation, we get $a \star b = \sum\limits_{(b)} b_{(2)} \Fo(a)\left(S(b_{(1)})\right)$ in $C(\mathbb{G})$.

Here, we restrict ourselves to random walks arising from the Fourier transform of linear combination of irreducible characters, it means from elements of the central algebra. This work is separated into two parts. The second section is devoted to the study of the Kac-Paljutkin group $\KP$. In the last one, we look at the Sekine family $\KP_n$. We study cut-off phenomenon in subsection \ref{cutoff} and list all the possible types of asymptotic behavior for $n$ fixed in Theorem \ref{BigTh}. It allows us to determine all the central idempotent states on these finite quantum groups, by noting that the central algebra is the center of the algebra with respect to the convolution product.

\section{Random walks in $\KP$}

First let us look at random walks on the eight-dimensional Kac-Paljutkin finite quantum group $\KP$. This is the smallest Hopf-von Neumann algebra which is neither commutative nor cocommutative. We use the notations given in \cite{arXiv1}. In particular, the multi-matrix algebra $C(\KP) = \CC \oplus \CC \oplus \CC \oplus \CC \oplus M_2(\CC)$ is endowed with the canonical basis $\{e_1, e_2, e_3, e_4, E_{11}, E_{12}, E_{21}, E_{22}\}$, where $E_{ij}$ is the image of a matrix unit. The unit is $\indic = 1 \oplus 1 \oplus 1 \oplus 1 \oplus I_2$, with $I_2$ the identity matrix. The Haar state  is given by
\[\hw \left(x_1 \oplus x_2 \oplus x_3 \oplus x_4 \oplus \begin{pmatrix}c_{11} & c_{12}\\c_{21} & c_{22}\end{pmatrix}\right) = \frac{1}{8} \left( x_1 + x_2 + x_3 + x_4 + 2 (c_{11} + c_{22}) \right) \text{ .}\]

As already noted by Kac and Paljutkin \cite{KPtrad}, the elements of $I(\KP)$ are
\begin{align*}
\rho(1) &= 1 \oplus 1 \oplus 1 \oplus 1 \oplus \begin{pmatrix}1&0\\0&1\end{pmatrix} \text{ ,}\\
\rho(2) &= 1 \oplus -1 \oplus -1 \oplus 1 \oplus \begin{pmatrix}1&0\\0&-1\end{pmatrix} \text{ ,}\\
\rho(3) &= 1 \oplus -1 \oplus -1 \oplus 1 \oplus \begin{pmatrix}-1&0\\0&1\end{pmatrix} \text{ ,}\\
\rho(4) &= 1 \oplus 1 \oplus 1 \oplus 1 \oplus \begin{pmatrix}-1&0\\0&-1\end{pmatrix} \text{ and}
\end{align*}
\[\X = \begin{pmatrix}
1 \oplus -1 \oplus 1 \oplus -1 \oplus \begin{pmatrix}0&0\\0&0\end{pmatrix}&
0 \oplus 0 \oplus 0 \oplus 0 \oplus \begin{pmatrix}0&e^{\imath \frac{\pi}{4}}\\e^{-\imath \frac{\pi}{4}}&0\end{pmatrix}\\
0 \oplus 0 \oplus 0 \oplus 0 \oplus \begin{pmatrix}0&e^{-\imath \frac{\pi}{4}}\\e^{\imath \frac{\pi}{4}}&0\end{pmatrix}&
1 \oplus 1 \oplus -1 \oplus -1 \oplus \begin{pmatrix}0&0\\0&0\end{pmatrix}
\end{pmatrix}\]
which gives the character $\chi(\X) = 2 \oplus 0 \oplus 0 \oplus -2 \oplus \begin{pmatrix}0&0\\0&0\end{pmatrix}$. Note that $\rho(1)$ is the unit $\indic$ in $C(\KP)$, this is the trivial representation of $\KP$.

Let us fix an element $g = \sum\limits_{u = 1} ^4 g_u \rho(u) \;+ g_\X \chi(\X)$ of $C(\KP)$, with $g_\alpha \in \CC$ for every $\alpha \in I(\KP)$. Let us note that $\Fo(e_i)(e_k) = \frac{1}{8} \delta_{ik}$, $\Fo(e_i)(E_{kl}) = 0$, $\Fo(E_{ij})(e_k) = 0$ and $\Fo(E_{ij})(E_{kl}) = \frac{1}{4} \delta_{il} \delta_{jk}$. Then, $\Fo(g)$ is a state on $C(\KP)$ if and only if $\Fo(g)(\indic) = 1$ and $\Fo(g)$ is positive, which means that $\Fo(g)(xe_i) \geq 0$ and $\Fo(g)(\sum\limits_{i,j = 1}^2 a_{ij} E_{ij}) \geq 0$ for any non negative real number $x$, any $1 \leq i \leq 4$ and any positive semidefinite matrix $A = (a_{ij})_{1 \leq i,j \leq 2} \in M_2(\CC)$. These conditions are equivalent to:
\begin{align}
g_1 = 1& \nonumber\\
1 + g_2 + g_3 + g_4 \pm& 2 g_{\X} \geq 0 \nonumber\\
1 + g_2 - g_3 - g_4 &\geq 0 \label{HKP}\\
1 - g_2 + g_3 - g_4 &\geq 0 \nonumber\\
1 - g_2 - g_3 + g_4 &\geq 0 \nonumber
\end{align}
In particular, it implies that, for all $\alpha \in I(\KP)$, $g_\alpha$ is real and moreover
\begin{equation}
\modu{g_\alpha} \leq d_\alpha \text{ .}
\label{CsqKP}
\end{equation}

\subsection{Upper and lower bounds}
\begin{lemma}
For all $u \in \{2,3,4\}$ and all positive integer $k$, 
\begin{equation}
\label{KPBd}
\frac{1}{2} \modu{g_u}^k \leq \norm{\Fo(g)^{\star k} - \hw\,}_{QTV} \leq \sqrt{\frac{1}{4} \sum_{v = 2}^4 \modu{g_v}^{2k} + \frac{\modu{g_\X}^{2k}}{4^{k}}}
\end{equation}
\label{LemBd}
\end{lemma}

\begin{proof}
The Quantum Diaconis-Shahshahani Upper Bound Lemma \cite[Lemma 5.3.8]{McCarthy} gives
\[\norm{\Fo(g)^{\star k} - \hw\,}_{QTV}^2 \leq \frac{1}{4} \sum_{\substack{\beta \in I(\KP)\\ \beta \neq \indic_\KP}} d_\beta \Tr\left(\widehat{\Fo(g)}(\beta)^{*k}\widehat{\Fo(g)}(\beta)^k\right)\]
with $\widehat{\Fo(g)}(\beta)$ the matrix composed by the $\Fo(g)(\beta_{ij})$'s for $\beta_{ij}$ the matrix coefficients of $\beta$.
On the other hand, by the Lower Bound Lemma \cite[Lemma 5.3.9]{McCarthy}, for all non trivial one-dimensional representation $\alpha$,
\[\norm{\Fo(g)^{\star k} - \hw\,}_{QTV} \geq \frac{1}{2} \modu{\widehat{\Fo(g)}(\alpha)}^k \text{ .}\]

Let us compute $\widehat{\Fo(g)}(\beta)$ for every $\beta \in I(\KP)$. First of all, let us note that, by linearity of $\Fo$,
\[\widehat{\Fo(g)}(\beta) = \sum_{\alpha \in I(\KP)}\hspace{-8pt}g_\alpha \widehat{\Fo(\chi_\alpha)}(\beta) = \left(\sum_{\alpha \in I(\KP)}\hspace{-8pt}g_\alpha \hw \beta_{ij} \chi_\alpha \right)_{1 \leq i,j \leq d_\beta} \text{ .}\]

Since the quantum group is finite, it is of Kac type. Hence the orthogonality relations for coefficients of irreducible representations lead to
\[\hw \beta_{ij} \chi_\alpha = \frac{\delta_{\alpha, \bar{\beta}}}{d_\alpha}\]
where $\bar{\beta} = \left((\beta_{ij})^*\right)_{1 \leq i,j \leq d_\beta}$. We clearly have $\bar{\beta} = \beta^{-1}$ when $d_\beta = 1$ and $\bar{\X} = \X$. Finally, we get $\widehat{\Fo(g)}(\beta) = g_{\beta^{-1}}$ when $d_\beta = 1$ and $\widehat{\Fo(g)}(\X) = \frac{1}{2} g_\X I_2$. Therefore, for $d_\beta = 1$,
\begin{align*}
&\Tr\left(\widehat{\Fo(g)}(\beta)^{*k}\widehat{\Fo(g)}(\beta)^k\right) = \modu{g_{\beta^{-1}}}^{2k}\\
\Tr&\left(\widehat{\Fo(g)}(\X)^{*k}\widehat{\Fo(g)}(\X)^k\right) = \frac{1}{2^{2k-1}} \modu{g_\X}^{2k} \text{ .}
\end{align*}
Hence, the Quantum Diaconis-Shahshahani Upper and Lower Bound Lemmas give the bounds announced in \eqref{KPBd}. \qed
\end{proof}

\subsection{Asymptotic behavior}

We can now determine the conditions for the random walk to converge to the Haar state. We then study the other cases, when the random walk does not converge to the Haar state.

\begin{theorem}
The random walk defined by $g = \sum\limits_{u = 1} ^4 g_u \rho(u) \;+ g_\X \chi(\X)$, satisfying conditions \eqref{HKP}, converges to the Haar state if and only if
\begin{equation}
\forall \alpha \in I(\KP)\setminus\{\indic_\KP\},\  \modu{g_\alpha} < d_\alpha \text{ .}
\label{KPcv}
\end{equation}
In this case, the distance to the Haar state decreases exponentially fast from the first step on.
\end{theorem}

\begin{proof}
By the upper bound \eqref{KPBd} in Lemma \ref{LemBd}, the condition \eqref{KPcv} is sufficient.

By the lower bound \eqref{KPBd} in Lemma \ref{LemBd}, the convergence implies that $\modu{g_u}$ is strictly less than $1$, for each $u \in \{2, 3, 4\}$. Then, the second equation in \eqref{HKP} leads to $\modu{g_\X} < 2$, i.e. \eqref{KPcv} holds.

In the case of convergence, let us denote by $g_m$ the greatest element among $\modu{g_2}$, $\modu{g_3}$, $\modu{g_4}$ and $\frac{\modu{g_\X}}{2}$. Then, $g_m < 1$ and, by Lemma \ref{LemBd},
\[\norm{\Fo(g)^{\star k} - \hw\,}_{QTV} \leq \sqrt{\frac{1}{4} \sum_{u = 2}^4 \modu{g_u}^{2k} + \frac{\modu{g_\X}^{2k}}{4^{k}}} \leq \frac{\sqrt{7}}{2} g_m^k\]
which shows the exponential decay. \qed
\end{proof}

By the inequality \eqref{CsqKP}, the random walk does not converge to the Haar state if and only if there exists a non-trivial irreducible representation $\alpha$ such that $\modu{g_\alpha} = d_\alpha$, i.e.\ $g_\alpha \in \left\{-d_\alpha, d_\alpha\right\}$. Let us note that it still might converge, but to some other state. Moreover, the limit state, if it exists, is clearly an idempotent state.

Let us now investigate conditions for the random walk to converge to an idempotent state different from $\hw\,\,$. First of all, let us list the idempotent states on $\KP$, given by Pal\cite{pal}, i.e. states on $\KP$ satisfying $(\phi \otimes \phi)\circ\Delta =\colon \phi \star \phi = \phi$:
\begin{align*}
\phi_1 &= 8 \Fo(e_1)\\
\phi_2 &= 4 \Fo(e_1 + e_2)\\
\phi_3 &= 4 \Fo(e_1 + e_3)\\
\phi_4 &= 4 \Fo(e_1 + e_4)\\
\phi_5 &= 2 \Fo(e_1 + e_2 + e_3 + e_4)\\
\phi_6 &= 2 \Fo(e_1 + e_4 + E_{11})\\
\phi_7 &= 2 \Fo(e_1 + e_4 + E_{22})\\
\phi_8 &= \Fo(\indic_\KP) = \hw
\end{align*}
Then, by looking at the different cases for $g$ to satisfy \eqref{HKP} and not \eqref{KPcv}, we get:
\begin{proposition}
Assume that the random walk associated to $g$ does not converge to the Haar state $\phi_8$. Then, we have six possibilities.

If $g_\X = 2$, then $g_u = 1$ for all $1 \leq u \leq 4$ and $\Fo(g) = \phi_ 1$.

If $g_\alpha = 1$ for at least two non trivial one-dimensional representations in $I(\KP)$ and $\modu{g_\X} < 2$, then $g_u = 1$ for all $1 \leq u \leq 4$ and the random walk converges to $\phi_4$.

If $g_2 = 1$, then $g_3 = g_4$. Assume also that $\modu{g_4} < 1$ and $\modu{g_\X} < 2$, then the random walk converges to $\phi_6$.

If $g_3 = 1$, then $g_2 = g_4$. Assume also that $\modu{g_2} < 1$ and $\modu{g_\X} < 2$, then the random walk converges to $\phi_7$.

If $g_4 = 1$, then $g_2 = g_3$. Assume also that $\modu{g_3} < 1$ and $\modu{g_\X} < 2$, then the random walk converges to $\phi_5$.

In all the other cases, it means if $g_\X = -2$ or if there is $2 \leq u \leq 4$ such that $g_u = -1$, then the random walk is cyclic and does no converge.
\end{proposition}

\begin{proof}
First, let us note that, by \cite[Proposition 2.1]{cutOffJP},
\[\norm{\Fo(g) - \Fo(x)}_{QTV} = \frac{1}{2} \norm{g - x}_1 = \frac{1}{2} \hw \modu{g - x} \text{ .}\]
We can deduce that the random walk associated to $g$ converges to $\phi_i$ if and only if the coefficients in front of $e_1$, $e_2$, $e_3$, $e_4$, $E_{11}$ and $E_{22}$ of $g^{\star k}$ converge to the corresponding ones in the definition of $\phi_i$ above.

Let us also note that $g^{\star k} = \displaystyle\sum_{\alpha \in I(\KP)} \frac{g_\alpha^k}{d_\alpha^{k-1}} \chi_\alpha$. In particular, the coefficients in front of $e_2$ and $e_3$ in $g^{\star k}$ are equal. Hence the random walk defined by such a $g$ cannot converge to $\phi_2$ or $\phi_3$.

Let us fix $g$ satisfying \eqref{HKP} and not \eqref{KPcv}. Assume, for instance $g_\X = 2$. Then, the inequality $1 + g_2 + g_3 + g_4 - 2 g_{\X} \geq 0$ in \eqref{HKP} implies that $g_u \geq 1$ for all $1 \leq u \leq 4$. Thus, by \eqref{CsqKP}, $g_u = 1$ for all $1 \leq u \leq 4$ and then $\Fo(g) = \phi_ 1$. Likewise, if $g_\X = -2$, $g_u = 1$ for all $1 \leq u \leq 4$ and $g^{\star k} = 8e_4$ or $8e_1$, depending on the parity of $k$. The other cases with $\modu{g_\X} < 2$ follow from a similar reasoning.\qed
\end{proof}

\section{Random walks in $\KP_n$}

We use the definition and the representation theory of the Sekine family of finite quantum groups presented in \cite{arXiv1}. Let us recall that for each $n$ greater than or equal to $2$, $C(\KP_n) = \bigoplus\limits_{i,j \in \Zn} \CC e_{(i,j)} \oplus M_n(\CC)$ is endowed with the Haar state $\hwn$ given by
\[\hw \left(\sum_{i,j \in \Zn} x_{(i,j)}e_{(i,j)} + \hspace{-5pt}\sum_{1 \leq i,j \leq n} X_{i,j} E_{i,j} \right) = \frac{1}{2n^2} \left(\sum_{i,j \in \Zn} x_{(i,j)} + n \sum_{i = 1}^n X_{i,i}\right)\]
where $E_{i,j}$ denotes a choice of matrix units in $M_n(\CC)$ and $x_{(i,j)}, X_{i,j} \in \CC$.

The set of irreducible representations $I(\KP_n)$ depends on the parity of $n$. For $n$ odd, its elements are, for $l \in \Zn$, the one-dimensional representations
\[\rho_l^\pm = \sum_{i, j \in \Zn} \eta^{il} e_{(i,j)} \pm \sum_{m = 1}^n E_{m, m+l}\]
with $\eta = e^{\frac{2 \imath \pi}{n}}$ a primitive $n^{\text{th}}$ root of unity, and the two-dimensional representations $\Xuv{u}{v}$, whose characters are
\[\chi\left(\Xuv{u}{v}\right) = 2 \sum_{s,t \in \Zn} \eta^{su} \cos\left(\frac{2tv\pi}{n}\right) e_{(s,t)}\]
for $u \in \Zn$ and $v \in \{1, 2, \ldots, \left\lfloor \frac{n-1}{2}\right\rfloor\}$.
If $n$ is even, we also need to add, for each $l \in \Zn$,
\[\sigma_l^\pm = \sum_{i, j \in \Zn} (-1)^j\eta^{il} e_{(i,j)} \pm \sum_{m = 1}^n (-1)^mE_{m, m+l} \text{ .}\]
Note that $\rho_0^+$ is the unit in $C(\KP_n)$, for each $n$.

Let us fix an element $a =\hspace{-5pt}\sum\limits_{\alpha \in I(\KP_n)}\hspace{-5pt}a_\alpha \chi_\alpha$ of $C(\KP_n)$. Let us denote $a_{(i,j)}$ and $A_{ij}$ its coefficients in the canonical basis. Then, with $\indic_{2\NN}$ the indicator function of even integers,
\[a_{(i,j)} = \sum_{l = 0}^{n-1}\hspace{-4pt}\left(\hspace{-2pt}a_{\rho_l^+}\hspace{-2pt}+ a_{\rho_l^-}\hspace{-2pt}+\hspace{-2pt}\indic_{2\NN}(n) (-1)^j(a_{\sigma_l^+}\hspace{-2pt}+ a_{\sigma_l^-})\hspace{-2pt}+ 2\hspace{-5pt}\sum_{v = 1}^{\lfloor \frac{n-1}{2} \rfloor}\hspace{-5pt}a_{\Xuv{l}{v}} \cos\left(\frac{2 \pi j v}{n} \right)\hspace{-4pt}\right) \eta^{il}\]
and $A_{i,j}$ is given by $a_{\rho_{j-i}^+} - a_{\rho_{j-i}^-} + \indic_{2\NN}(n) (-1)^i(a_{\sigma_{j-i}^+} - a_{\sigma_{j-i}^-})$.

\begin{lemma}
\label{HaPos}
The functional $\Fo(a)$ is a state if and only if $a_{\rho_0^+} = 1$, for all $i, j$ in $\Zn$, $
a_{(i,j)} \geq 0$ and the matrix $A = (A_{ij})_{1 \leq i,j \leq n}$ is positive semi-definite
\end{lemma}
 
\begin{proof}
We want $\Fo(a)$ to satisfy the first conditions of \cite[Lemma 6.4]{idem}. To obtain the result, we only need to note that
\[2n^2 \Fo(e_{(s,t)})(e_{(i,j)}) = \delta_{s,i} \delta_{t,j} \qquad 2n^2 \Fo(e_{(s,t)})(E_{i,j}) = 0\]
\[2n \Fo(E_{p,q})(e_{(i,j)}) = 0 \qquad 2n \Fo(E_{p,q})(E_{i,j}) = \delta_{p,j} \delta_{q,i} \text{ .}\]\qed
\end{proof}

\begin{remark}
By Lemma \ref{HaPos}, the coefficient $a_{\rho_0^-}$ is real and its absolute value is not greater than 1.
\end{remark}

\begin{example}
\label{ex1}
Let $p, q$ be two integers such that $pq = n$, $q > 1$. Put $a_{\rho_{lp}^+} = \eta^{lp}$ for all $0 \leq l \leq q-1$ and $a_\alpha = 0$ for all the others $\alpha \in I(\KP_n)$. Then, for all $i, j \in \Zn$
\[a_{(i,j)} = \sum_{l = 0}^{q-1} \eta^{lp} \eta^{ilp} = q \indic_{q\Zn}(i+1)\]
which equals $q$ if $i+1$ is a multiple of $q$ in $\Zn$ and $0$ otherwise. This is non negative, and the matrix $A$, defined by $A_{ij} = \eta^{lp}$ if $j = i + lp$ and $0$ otherwise, is self-adjoint and its eigenvalues are $0$ and $q$, i.e.\ $A$ is semidefinite positive. Therefore $\Fo(a) = \phi_p$ is a state on $C(\KP_n)$.
\end{example}

\subsection{Conditions for convergence}
First, note that, if we extend naturally the notation, $\Xuv{u}{0}$ is unitarily equivalent to the direct sum of the representations $\rho_u^+$ and $\rho_u^-$. Similarly, if $n$ is even, $\Xuv{u}{\frac{n}{2}}$ is unitarily equivalent to the direct sum of $\sigma_u^+$ and $\sigma_u^-$, as representations of $\KP_n$.

Moreover, the random walk associated to $a$ converges if and only if the sequences of matrices $\left( \widehat{\Fo(a)}(\Xuv{u}{v})^k \right)_{k \in \NN}$ converge for all $u \in \Zn$ and all $v \in \{0, 1, \ldots, \left\lfloor \frac{n}{2} \right\rfloor\}$, since
\[\widehat{(\Fo(a)^{\star k})} (\Xuv{u}{v}) = \left(\widehat{\Fo(a)}(\Xuv{u}{v})\right)^k \text{ .}\]

\begin{theorem}
\label{CondCv}
The random walk associated to $a$ converges if and only if every complex number $\frac{a_\alpha}{d_\alpha}$ is an element of the open unit disc $\mathbb{D}$ or equals $1$, for all $\alpha \in I(\KP_n)$. 
\end{theorem}

\begin{proof}
First, let us fix $(u,v)$ such that $0 \leq u \leq n-1$ and $1 \leq v \leq \left \lfloor \frac{n-1}{2} \right \rfloor$ . Then
\[\widehat{\Fo(a)}(\Xuv{u}{v}) = \begin{pmatrix} \sum\limits_{s,t = 0} ^{n-1} \eta^{su+tv} \hwn e_{(s,t)} a & \sum\limits_{m = 1} ^{n} \eta^{mv} \hwn E_{m, m+u} a\\ \sum\limits_{m = 1} ^{n} \eta^{-mv} \hwn E_{m, m+u} a &  \sum\limits_{s,t = 0} ^{n-1} \eta^{su-tv} \hwn e_{(s,t)} a\end{pmatrix}\]
is easy to calculate if we note that $\hwn e_{(i,j)} a = \frac{a_{(i,j)}}{2n^2}$, and $\hwn E_{m, m+u} a = \frac{1}{2n} A_{m+u,m}$. By direct calculations, we obtain $\widehat{\Fo(a)}(\Xuv{u}{v}) = \frac{1}{2} a_{\Xuv{n-u}{v}} I_2$.
Hence the corresponding sequences of matrices converge if and only if
\[\frac{1}{2}\modu{a_{\Xuv{n-u}{v}}} < 1 \text{ or } \frac{1}{2} a_{\Xuv{n-u}{v}} = 1 \text{ .}\]

The same computation for $v = 0$ gives
\begin{align*}
\widehat{\Fo(a)}(\Xuv{u}{0}) &= \frac{1}{2}\begin{pmatrix}a_{\rho_{n-u}^+} + a_{\rho_{n-u}^-} & a_{\rho_{n-u}^+} - a_{\rho_{n-u}^-} \\ a_{\rho_{n-u}^+} - a_{\rho_{n-u}^-} & a_{\rho_{n-u}^+} + a_{\rho_{n-u}^-}\end{pmatrix}\\
&= \begin{pmatrix}\frac{1}{\sqrt{2}}&\frac{1}{\sqrt{2}}\\\frac{1}{\sqrt{2}}&-\frac{1}{\sqrt{2}}\end{pmatrix} \begin{pmatrix}a_{\rho_{n-u}^+}&0\\0&a_{\rho_{n-u}^-}\end{pmatrix}\begin{pmatrix}\frac{1}{\sqrt{2}}&\frac{1}{\sqrt{2}}\\\frac{1}{\sqrt{2}}&-\frac{1}{\sqrt{2}}\end{pmatrix} \text{ .}
\end{align*}
This leads to the condition on $a_{\rho_{l}^+}$ and $a_{\rho_{l}^-}$.

If $n$ is even, we get
\[\widehat{\Fo(a)}(\Xuv{u}{\frac{n}{2}}) = \frac{1}{2}\begin{pmatrix}a_{\sigma_{n-u}^+} + a_{\sigma_{n-u}^-} & (-1)^u\left(a_{\sigma_{n-u}^+} - a_{\sigma_{n-u}^-}\right) \\ (-1)^u\left(a_{\sigma_{n-u}^+} - a_{\sigma_{n-u}^-}\right) & a_{\sigma_{n-u}^+} + a_{\sigma_{n-u}^-}\end{pmatrix}\]
whose eigenvalues are $a_{\sigma_{n-u}^+}$ and $a_{\sigma_{n-u}^-}$, and the condition on these coefficients follows. \qed
\end{proof}

\begin{remark}
\label{rkMod}
Since the $\Xuv{u}{v}$'s are unitaries, the $\widehat{\Fo(a)}(\Xuv{u}{v})$'s have norm not greater than $1$, when $\Fo(a)$ is a state. In particular, for each irreducible representation $\alpha$ in $I(\KP_n)$, the dimension bounds the modulus of the corresponding coefficient, i.e.\ $\modu{a_\alpha} \leq d_\alpha$.
\end{remark}

\subsection{Distance to uniformity}

Now that we have a criterion for convergence, we want to single out conditions ensuring that the limit is the Haar state. As in the case of the Kac-Paljutkin finite quantum group, we look at convergence to the Haar state by using Quantum Diaconis-Shahshahani Upper and Lower Bound Lemmas.

\subsubsection{Upper bound}
\begin{lemma}
For any natural integer $k$,
\begin{equation}
\norm{\Fo(a)^{\star k} - \hwn\,}_{QTV} \leq \sqrt{\frac{1}{4}\hspace{-5pt}\sum_{\substack{\alpha \in I(\KP_n)\\ \alpha \neq \rho_0^+,\,d_\alpha = 1}}\hspace{-13pt}\modu{a_\alpha}^{2k} + \frac{1}{4^k}\sum_{\substack{\alpha \in I(\KP_n)\\d_\alpha = 2}} \hspace{-5pt}\modu{a_{\Xuv{u}{v}}}^{2k}}
\label{UpBd}
\end{equation}
\end{lemma}

\begin{proof}
The Quantum Diaconis-Shahshahani Upper Bound Lemma \cite[Lemma 5.3.8]{McCarthy} gives
\[\norm{\Fo(a)^{\star k} - \hwn\,}_{QTV}^2 \leq \frac{1}{4} \sum_{\substack{\beta \in I(\KP_n)\\ \beta \neq \rho_0^+}} d_\beta \Tr\left(\widehat{\Fo(a)}(\beta)^{*k}\widehat{\Fo(a)}(\beta)^k\right) \text{ .}\]

Let us compute $\widehat{\Fo(a)}(\beta)$ for every $\beta \in I(\KP_n)$. First of all, let us recall that, by linearity of $\Fo$,
\[\widehat{\Fo(a)}(\beta) = \sum_{\alpha \in I(\KP_n)}\hspace{-8pt}a_\alpha \widehat{\Fo(\chi_\alpha)}(\beta) = \left(\sum_{\alpha \in I(\KP_n)}\hspace{-8pt}a_\alpha \hwn \beta_{ij} \chi_\alpha \right)_{1 \leq i,j \leq d_\beta} \text{ .}\]

Assume $d_\beta = 1$. By the orthogonality relations for coefficients of irreducible representations of quantum group of Kac type, we get $\widehat{\Fo(a)}(\beta) = a_{\beta^{-1}}$ and then
\begin{equation}
\label{eq3}
\Tr\left(\widehat{\Fo(a)}(\beta)^{*k}\widehat{\Fo(a)}(\beta)^k\right) = \modu{a_{\beta^{-1}}}^{2k} \text{ .}
\end{equation}

Now, assume $d_\beta = 2$, which means that there exists $(u,v)$ such that $\beta = \Xuv{u}{v}$, $0 \leq u \leq n-1$ and $1 \leq v \leq \left \lfloor \frac{n-1}{2} \right \rfloor$. We already know that $\widehat{\Fo(a)}(\Xuv{u}{v}) = \frac{1}{2} a_{\Xuv{n-u}{v}} I_2$. Therefore
\begin{equation}
\label{eq4}
\Tr\left(\widehat{\Fo(a)}(\Xuv{u}{v})^{*k}\widehat{\Fo(a)}(\Xuv{u}{v})^k\right) = \frac{1}{2^{2k-1}} \modu{a_{\Xuv{n-u}{v}}}^{2k} \text{ .}
\end{equation}

Finally, we get the upper bound \eqref{UpBd} by replacing \eqref{eq3} and \eqref{eq4} in the Quantum Diaconis-Shahshahani Upper Bound Lemma. \qed
\end{proof}

Hence, the random walk associated to $a$ converges to the Haar state if all the complex numbers cited in Theorem \ref{CondCv} are not equal to $1$, except $a_{\rho_0^+}$, which means that
\begin{equation}
\label{Hcv0}
\forall \alpha \in I(\KP_n)\setminus \{\rho_0^+\}, \ \modu{a_\alpha} < d_\alpha
\end{equation} 

\subsubsection{Lower bound}
\begin{lemma}
For any non trivial irreducible representation $\alpha$, 
\begin{equation}
\frac{1}{2} \modu{\frac{a_\alpha}{d_\alpha}}^k \leq \norm{\Fo(a)^{\star k} - \hwn\,}_{QTV} \text{ .}
\label{LoBd}
\end{equation}
\end{lemma}

\begin{proof}
By the Quantum Diaconis-Shahshahani Lower Bound Lemma \cite[Lemma 5.3.9]{McCarthy}, for all non trivial one-dimensional representation $\alpha$,
\[\norm{\Fo(a)^{\star k} - \hwn\,}_{QTV} \geq \frac{1}{2} \modu{a_\alpha}^k \text{ .}\]
Moreover, the following equality \cite{McCarthy} holds
\begin{equation}
\norm{\Fo(a)^{\star k} - \hwn\,}_{QTV} = \frac{1}{2} \sup\limits_{x \in C(\KP_n), \norm{x}_\infty \leq 1} \modu{\Fo(a)^{\star k}(x) - \hwn x}
\label{eqDefQTV}
\end{equation}
and any two-dimensional irreducible representation satisfies $\norm{\frac{\chi(\Xuv{u}{v})}{2}}_\infty \leq 1$. Thus
\[\norm{\Fo(a)^{\star k} - \hwn\,}_{QTV} \geq \frac{1}{4} \modu{\Fo(a)^{\star k}\left(\chi(\Xuv{u}{v})\right)} = \frac{1}{2} \modu{\frac{a_{\Xuv{n-u}{v}}}{2}}^k \text{ .}\]
\qed
\end{proof}

Finally, combining the previous lemmas, we obtain
\begin{theorem}
The random walk $(\Fo(a)^{\star k})_{k \in \NN}$ converges to the Haar state if and only if the condition \eqref{Hcv0} holds.
\label{ThHaar}
\end{theorem}

\begin{proof}
The condition \eqref{Hcv0} is clearly sufficient for the the random walk to convergence to the Haar state, thanks to the upper bound \eqref{UpBd}. Moreover, if \eqref{Hcv0} is not satisfied for some non trivial representation, by \eqref{LoBd}, we obtain that the sequence $\left(\Fo(a)^{\star k}\right)_{k \in \NN}$ stays "uniformly far away from randomness", since the norm is at least $\frac{1}{2}$ for all $k$. Thus the condition \eqref{Hcv0} is also necessary.
\end{proof}

\begin{remark}
Let us note that if the condition \eqref{Hcv0} is not satisfied for $\Xuv{u}{v}$, $a_{\Xuv{u}{v}} = 2$, by Theorem \ref{CondCv}. Hence
\[\widehat{\Fo(a)}(\Xuv{n-u}{v}) = \frac{1}{2} a_{\Xuv{u}{v}} I_2 = I_2\]
and the limit of the sequence $\left( \widehat{\Fo(a)}(\Xuv{n-u}{v})^k \right)_{k \in \NN}$ is the identity. But, $\widehat{\hwn\;}(\Xuv{p}{q})$ is the matrix null except if $p=q=0$, it means that the limit state is not $\hwn\;$. This is another way to prove that the condition \eqref{Hcv0} is necessary for the two-dimensional representations. We will use this method to study the asymptotic behavior of the random walk.
\end{remark}

\subsection{Cut-off phenomenon}
\label{cutoff}

In the classical case, we frequently observe cut-off phenomenon, i.e.\ the distance between $\Fo(a)^{\star k}$ and the Haar state is almost $1$ for a time and then suddenly tends to $0$, exponentially fast. There exist different definitions of the threshold (see for example \cite{SaloffCoste},\cite{DiaconisDefCO} and \cite{roussel}), depending on how sharp you want the cut-off.

Note that, by upper and lower bounds, \eqref{UpBd} and \eqref{LoBd}, we have roughly
\[\frac{1}{2} M_a^k \leq \norm{\Fo(a)^{\star k} - \hwn\,}_{QTV} \leq \frac{n}{\sqrt{2}} M_a^k\]
where $M_a$ is the maximum of the $\frac{\modu{a_\alpha}}{d_\alpha}$, $\alpha$ non-trivial. For $n$ fixed, if $M_a$ is less than $1$, the random walk converges to the Haar state. In this subsection, we will look at what happens if $a$ and $k$ depend on $n$. To give a definition of $a$ depending quite naturally on $n$, we look at some quantum version of the simple random walk on $\Zn$, studied in \cite[Theorem 2, Chapter 3C]{DiaconisExCO}.

If we only consider these bounds and want that $n M_a^k$ vanishes when $n$ goes to infinity, then the limit of $M_a^k$ is also $0$. Moreover, since the lower bound is at most $\frac{1}{2}$, it will be ineffective to prove the cut-off. Actually, for our example, we are able to prove that there is no cut-off, as in its classical version.

\begin{remark}
In \cite{cutOffJP}, McCarthy gives another example of a random walk in Sekine family where there is no cut-off, but with a state which formally does not depend on $n$ and which is not the Fourier transform of an element in the central algebra. Note that in \cite{cutOffOn} and \cite{cutOffUn}, Freslon finds cut-off examples in compact quantum groups.
\end{remark}

\begin{theorem}
For $n$ odd, define the element $a \in C(\KP_n)$ such that for any $l \in \Zn$, $a_{\rho_l^+} = \cos\left(\frac{2l\pi}{n}\right)$ and $a_\alpha = 0$ otherwise. Then $\Fo(a)$ defines a state on $C(\KP_n)$, such that, for any non negative $c$, if $k = e^cn^2$,
\[\frac{1}{2} e^{-k\left(\frac{\pi^2}{2n^2}+ \frac{\pi^4}{4n^4}\right)} \leq\norm{\Fo(a)^{\star k} - \hwn\,}_{QTV} \leq e^{-k\frac{\pi^2}{2n^2}} \text{ .}\]
Moreover the lower bound still holds when $k = e^{-c}n^2$.
\end{theorem}

\begin{remark}
If $n$ is even, the same definition for $a_{\rho_l^+}$ leads to $a_{\rho_{\frac{n}{2}}^+} = -1$ and by Theorem \ref{ThHaar} the random walk does not converge to the Haar state.
\end{remark}

\begin{proof}
First let us prove that $\Fo(a)$ is a state on $C(\KP_n)$. We only need to check that every $a_{(i,j)}$ is non negative and that the matrix $A$ is semidefinite positive, since $a_{\rho_0^+}$ is equal to $1$ by definition.

Note that, for any $i$ and any $j$ in $\Zn$,
\[a_{(i,j)} = \sum_{l=0}^{n-1} \cos\left(\frac{2l\pi}{n}\right) \eta^{il} = \frac{n}{2}\left( \delta_{i,1} + \delta_{i,n-1}\right)\]
is always non negative and the equality $A_{i,j} = \cos\left(\frac{2(j-i)\pi}{n}\right)$ defines a circulant matrix, whose eigenvalues are given by the same formula as the $a_{(i,j)}$'s, for $i$ between $0$ and $n-1$. Thus, the matrix $A$ is semidefinite positive.

On the one hand, following the arguments of the proof of Theorem 2 in Chapter 3C of \cite{DiaconisExCO}, the upper bound \eqref{UpBd} gives
\begin{align*}
\norm{\Fo(a)^{\star k} - \hwn\,}_{QTV}^2 &\leq \frac{1}{4} \sum_{l=1}^{n-1} \cos\left(\frac{2l\pi}{n}\right)^{2k}\\
&\leq \frac{1}{2} \sum_{l=1}^{\frac{n-1}{2}} \cos\left(\frac{l\pi}{n}\right)^{2k}\\
&\leq \frac{1}{2} \sum_{l=1}^{\frac{n-1}{2}} e^{-\frac{k l^2 \pi^2}{n^2}}
\end{align*}
where we use that $\cos(x) \leq e^{-\frac{x^2}{2}}$ for $0 \leq x < \frac{\pi}{2}$. Finally, since for every integer $j$ we have $(j+1)^2-1 \geq 3j$,
\begin{align*}
\norm{\Fo(a)^{\star k} - \hwn\,}_{QTV}^2 &\leq \frac{e^{-\frac{k \pi^2}{n^2}}}{2} \sum_{j=0}^{+\infty} e^{-\frac{3k j \pi^2}{n^2}}\\
& \hspace{15pt}= \frac{e^{-\frac{k \pi^2}{n^2}}}{2 \left(1 - e^{-\frac{3k \pi^2}{n^2}}\right)} \text{ .}
\end{align*}
This works for any $k$ and any odd $n$. Moreover, note that, for $k$ greater than $n^2$, the denominator $2 \left(1 - e^{-\frac{3k \pi^2}{n^2}}\right)$ is greater than $1$, thus we obtain the announced upper bound.

On the other hand, the lower bound \eqref{LoBd} leads to
\begin{align*}
\norm{\Fo(a)^{\star k} - \hwn\,}_{QTV} \geq& \frac{1}{2} \max\limits_{0 \leq l \leq n-1} \modu{\cos\left(\frac{2\pi l}{n}\right)}^k\\
&= \frac{1}{2} \cos\left(\frac{\pi}{n}\right)^k \text{ .}
\end{align*}
Let us define $g(x) = \log\left(e^{\frac{x^2}{2} + \frac{x^4}{4}}\cos(x)\right)$. Since $g^\prime(x) = x + x^3-\tan(x)$ is positive on $\left]0\;;\;\frac{\pi}{3}\right]$, $g(x)$ is also positive on the same interval. Hence
\[\norm{\Fo(a)^{\star k} - \hwn\,}_{QTV} \geq \frac{1}{2} \cos\left(\frac{\pi}{n}\right)^k \geq \frac{1}{2} e^{-\frac{k\pi^2}{2n^2}-\frac{k\pi^4}{4n^4}} \text{ .}\]
This works for any $k$ and any odd $n\geq 3$. \qed
\end{proof}

\begin{remark}
Thus, the limit of $\norm{\Fo(a)^{\star k(n)} - \hwn\,}_{QTV}$, when $n$ and $\frac{k(n)}{n^2}$ go to infinity, is null, and conversely, is always at least $\frac{1}{2}e^{-\frac{\pi^2}{2}-\frac{\pi^4}{36}}$ for $k(n)$ less than $n^2$.
\end{remark}

\begin{remark}
These bounds give in particular that for any positive integer $C$,
\[0 < \liminf\limits_{n \to + \infty} \norm{\Fo(a)^{\star Cn^2} - \hwn\,}_{QTV} \leq \limsup\limits_{n \to + \infty} \norm{\Fo(a)^{\star Cn^2} - \hwn\,}_{QTV} < 1 \text{ .}\]
Thus $k$ should satisfies $\lim\limits_{n \to +\infty} \frac{k}{n^2} = +\infty$ in order to get the convergence of the upper bound to $0$ when $n$ goes to infinity. But, in this case, the lower bound vanishes too. This is precisely the meaning of a ``no cut-off statement.
\end{remark}

\subsection{Asymptotic behavior}

We go back to the setting where $n$ and $a$ are fixed, and $k$ does not depend on them. We now study what happens when $k$ goes to infinity.

\subsubsection{Possible limit states}
Let us note that, when the random walk converges, the limit state $\mu$ is an idempotent state. Thanks to Zhang \cite{Haonan}, we know all the idempotent states on the Sekine quantum groups. There are four different types:
\begin{itemize}
	\item $\hwn$ , the Haar state,
	\item $h_\Gamma = \frac{2n^2}{\#\Gamma} \sum\limits_{(i,j) \in \Gamma} \hspace{-5pt}\Fo(e_{(i,j)})$ where $\Gamma$ is a subgroup of $\Zn \times \Zn$ ,
	\item $h_{\Gamma,l} =  \frac{n^2}{\#\Gamma} \sum\limits_{(i,j) \in \Gamma} \hspace{-5pt}\Fo(e_{(i,j)}) + q \sum\limits_{m \equiv l [q]} \Fo(E_{m,m})$ where $\Gamma = \Zn \times q\Zn$, $q \mid n$, $q > 1$ and $l \in \mathbb{Z}_q$ ,
	\item $h_{\Gamma,l, \tau} =  \frac{n^2}{\#\Gamma}\sum\limits_{(i,j) \in \Gamma}\hspace{-5pt}\Fo(e_{(i,j)}) + q\sum\limits_{i,j \equiv l [q]}\hspace{-3pt}\tau_{j-i}\Fo(E_{i,j})$ where $\Gamma = p\Zn \times q\Zn$, $p > 1$, $pq = n$, $l \in \mathbb{Z}_q$ and $\tau \in \{-1, 1\}^{q\Zn}$ such that $\sum\limits_{j \in q\Zn}\hspace{-5pt}\tau_j \eta^{ij} > 0$ for all $i \in \mathbb{Z}_p$.	
\end{itemize}

Let us note that the two first types are Haar idempotent states, which means that they arise from the Haar state of a quantum subgroup.

\begin{remark}
We already know that
\[\widehat{\hwn}(\Xuv{u}{v}) = \delta_{(u,v),(0,0)} I_2 \text{ .}\]
The Remark 2.11 in \cite{Haonan} gives a useful characterization of types $h_{\Gamma, l}$ and $h_{\Gamma, l, \tau}$ by the $\hat{\mu}(\Xuv{u}{v})$'s:
\begin{equation}
\widehat{h_{\Gamma, l}}(\Xuv{u}{v}) = \begin{cases}
\frac{1}{2} \begin{pmatrix} 1&\eta^{-vl}\\\eta^{vl}&1\end{pmatrix} &\text{ if } u = 0 \text{ and } \frac{n}{q} \mid v\\
\begin{pmatrix}0&0\\0&0\end{pmatrix}&\text{ otherwise}
\end{cases}
\label{CarhGl}
\end{equation}
\begin{equation}
\widehat{h_{\Gamma, l, \tau}}(\Xuv{u}{v}) = \begin{cases}
\frac{1}{2} \begin{pmatrix} 1&\tau_{-v}\eta^{-vl}\\\tau_v\eta^{vl}&1\end{pmatrix} &\text{ if } q \mid u \text{ and } \frac{n}{q} \mid v\\
\begin{pmatrix}0&0\\0&0\end{pmatrix}&\text{ otherwise}
\end{cases}
\label{CarhGlT}
\end{equation}
In particular, for $\mu$ of these types, and for each $v$ between $1$ and $\lfloor \frac{n-1}{2}\rfloor$, the matrix $\hat{\mu}(\Xuv{n-u}{v})$ is either null or not diagonal. Thus, when some $a_{\Xuv{u}{v}}$ equals $2$, by the proof of Theorem \ref{ThHaar}, the limit state should be of type $h_\Gamma$, if it exists.
\label{rk2}
\end{remark}

\begin{proposition}
If the random walk associated to $a$ converges, the limit state $\mu$ is a central idempotent state.
\label{IdemCent}
\end{proposition}

\begin{proof}
A central idempotent state is a state $\phi$ on $C(\KP_n)$ such that $\phi \star \phi = \phi$ and $\phi \star \psi = \psi \star \phi$ for every bounded linear functional $\psi$ on $C(\KP_n)$.

Let us note that if $\Fo(a)^{\star k} \star \psi = \psi \star \Fo(a)^{\star k}$ holds for all $k$, then $\mu \star \psi = \psi \star \mu$ is verified, where $\mu$ is the limit state of the random walk $\Fo(a)^{\star k}$. Thus, we will study the commutativity for the random walk. Actually, we only need to check if $\Fo(a)$ is central, for $a$ a linear combination of irreducible characters.

Let $\psi$ be a linear functional on $C(\KP_n)$. Then, there exists $b$ in $C(\KP_n)$ such that $\psi = \Fo(b)$. We want to prove the equality $\psi \star \Fo(a) = \Fo(a) \star \psi$. This is to show that if $a$ is a linear combination of characters, $a \star b = b \star a$ holds for all $b$ in $C(\KP_n)$.

Using Sweedler notation, by definition of the convolution product on $C(\KP_n)$, since the Haar state is tracial and $S$ is involutive,
\begin{align*}
b \star a &= \sum_{(a)} a_{(2)} \hwn S(a_{(1)})b\\
&= \left(\hwn\, \otimes \rho_0^+\right)\left[(b \otimes \rho_0^+)\, \cdot (S\otimes \rho_0^+)\cop(a) \right]\\
&= \left(\hwn\, \otimes S\right)\left[(b \otimes \rho_0^+)\, \cdot (S\otimes S)\cop(a) \right] \text{ .}
\end{align*}
Let us denote by $\sigma$ the morphism $x \otimes y \mapsto y \otimes x$, and note that for any $\alpha$ in $I(\KP_n)$, $\sigma(\cop(\chi_\alpha)) = \cop(\chi_\alpha)$.

The property of the antipode gives $(S\otimes S)\cop(a) = \sigma\left(\cop(S(a))\right)$. Moreover the element $S(a)$ is also a linear combination of characters, hence $\sigma\left(\cop\left(S(a)\right)\right) = \cop(S(a))$. Thus, we get
\begin{align*}
b \star a &= S\left(\hwn \otimes \rho_0^+\right)\left[(b \otimes \rho_0^+)\, \cdot \cop(S(a)) \right]\\
&= S^2\left(\hwn \otimes \rho_0^+\right)\left[\cop(b) \, \cdot (S(a) \otimes \rho_0^+) \right]\\
&=\left(\hwn \otimes \rho_0^+\right)\left(S \otimes \rho_0^+\right)\left[\left(S \otimes \rho_0^+\right)\cop(b) \, \cdot (a \otimes \rho_0^+) \right]\\
&= a \star b
\end{align*}
where the second equality comes from the Kac property, and the last one from the fact that $\hwn S(x) = \hwn x$. \qed
\end{proof}

\begin{remark}
This result is true in any finite quantum group since they satisfy the Kac property. Note that the Haar state is always invariant under the antipode on a compact quantum group \cite{idem}.
\end{remark}

\subsubsection{Convergence to $h_\Gamma$}
\begin{lemma}
Assume that the random walk associated to $a$ converges, and denote by $\mu$ the limit state. The followings are equivalent:
\begin{enumerate}[label=\roman*)]
	\item $\forall l \in \Zn, \ a_{\rho_l^+} = a_{\rho_l^-}$ and $a_{\sigma_l^+} = a_{\sigma_l^-}$ (when $n$ is even)
	\item $a_{\rho_0^-} = 1$
	\item $\mu$ is a $h_\Gamma$ for $\Gamma$ a subgroup of $\Zn \times \Zn$ such that
\begin{equation}
(k,l) \in \Gamma \ \ \Longleftrightarrow \ \ (k, -l) \in \Gamma
\label{eqSim}
\end{equation}
\end{enumerate}
\label{hGamma}
\end{lemma}

\begin{proof}
The implication $i) \Rightarrow ii)$ is clear. The converse follows from the fact that $A$ is positive semidefinite.

Now, let us look at the equivalence between $ii)$ and $iii)$. First of all, let us note that
\[a^{\star k} = \sum_{\alpha \in I(\KP_n)} \frac{a_\alpha ^k}{d_\alpha^{k-1}} \chi_\alpha \text{ .}\]
Let us denote by $a_{(i,j)}^{(k)}$ and $A_{i,j}^{(k)}$ its coefficients in the canonical basis, and by $\mu_{(i,j)}$ and $M_{i,j}$ the coefficients of $\mu$ in basis $\{\Fo(e_{(i,j)}), \Fo(E_{i,j}), 1 \leq i, j \leq n\}$.

Then by \cite{Haonan}, either $\sum\limits_{i,j \in \Zn}\hspace{-5pt}\mu_{(i,j)} = 2n^2$ and $\sum\limits_{m = 1}^n\hspace{-3pt}M_{m,m} = 0$ and $\mu$ is of type $h_\Gamma$, or $\sum\limits_{i,j \in \Zn}\hspace{-5pt}\mu_{(i,j)} = n^2$ and $\sum\limits_{m = 1}^n\hspace{-3pt}M_{m,m} = n$ and $\mu$ is of another type.

On the other hand,
\[\sum_{i,j \in \Zn}\mu_{(i,j)} = \lim_{k \to +\infty} \sum_{i,j \in \Zn} a_{(i,j)}^{(k)} = \lim_{k \to +\infty} n^2 \left( 1 + a_{\rho_0^-}^k\right)\]
\[\sum_{m = 1}^nM_{m,m} = \lim_{k \to +\infty} \sum_{i \in \Zn} A_{i,i}^{(k)} = \lim_{k \to +\infty} n \left( 1 - a_{\rho_0^-}^k\right)\]

Therefore the limit $\mu$ satisfies $\sum\limits_{i,j \in \Zn}\hspace{-5pt}\mu_{(i,j)} = 2n^2$ and $\sum\limits_{m = 1}^n\hspace{-3pt}M_{m,m} = 0$ if and only if $a_{\rho_0^-} = 1$.

Moreover, for $a$ a linear combination of irreducible characters, $a_{(i,-j)}^{(k)} = a_{(i,j)}^{(k)}$ then $\mu_{(i,j)} = \mu_{(i,-j)}$. Thus the subgroup $\Gamma$ of $\mathbb{Z}_n \times \mathbb{Z}_n$ satisfies the condition \eqref{eqSim}. \qed
\end{proof}

\begin{remark}
By Remark \ref{rk2}, we obtain that if the random walk converges and there exists $\alpha \in I(\KP_n)$ such that $a_\alpha = d_\alpha  = 2$, then $a_{\rho_0^-} = 1$.
\end{remark}

\begin{lemma}
Assume that the random walk associated to the element $a$ converges. If there exists a one-dimensional representation $\alpha^+$ in $I(\KP_n)$ such that $a_{\alpha^+} = a_{\alpha^-} = 1$, then the limit state is of type $h_\Gamma$, with $\Gamma$ a subgroup of $\Zn \times \Zn$ satisfying \eqref{eqSim}.
\label{lemEq}
\end{lemma}

\begin{proof}
Let us denote by $\X^{\alpha}$ the corresponding $\Xuv{u}{v}$ with $u \in \Zn$ and $v = 0$ if $\alpha^+$ is $\rho_{n-u}^+$ or $v = \frac{n}{2}$ if $\alpha^+ = \sigma_{n-u}^+$.
Then $\widehat{\Fo(a)^{\star k}}(\X^\alpha) = I_2$.  Hence, by characterizations \eqref{CarhGl} and \eqref{CarhGlT}, the limit state is of type $h_\Gamma$. The same argument as in the preceding lemma gives the property of $\Gamma$. \qed
\end{proof}

\subsubsection{Classification}
\begin{theorem}
Denote by $\mu$ the limit state of the random walk associated to $a$, if it exists.

If $n$ is odd, there are only four possibilities:
\begin{itemize}
	\item condition \eqref{Hcv0} holds, i.e. $\mu = \hwn$ ,
	\item $\forall \alpha \in I(\KP_n), \ \modu{a_\alpha} < d_\alpha \text{ or } a_\alpha = d_\alpha$, and $a_{\rho_0^-} = 1$, i.e. $\mu = h_\Gamma$, with $\Gamma$ a subgroup of $\Zn \times \Zn$ satisfying \eqref{eqSim}
	\item $\forall \alpha \in I(\KP_n), \ \modu{a_\alpha} < d_\alpha \text{ or } a_\alpha = 1$, $a_{\rho_0^-} \neq 1$ but condition \eqref{Hcv0} does not hold, i.e. $\mu = h_{\Gamma, l, \tau}$ with $q = 1$ (and $p = n$), $l = 0$ and $\tau_i = 1$ for all $i \in \Zn$,
	\item otherwise the random walk diverges.
\end{itemize}

If $n$ is even, there are six possibilities:
\begin{itemize}
	\item condition \eqref{Hcv0} holds, i.e. $\mu = \hwn$ ,
	\item $\forall \alpha \in I(\KP_n) \ \modu{a_\alpha} < d_\alpha \text{ or } a_\alpha = d_\alpha$, and $a_{\rho_0^-} = 1$, i.e. $\mu = h_\Gamma$, with $\Gamma$ a subgroup of $\Zn \times \Zn$ satisfying \eqref{eqSim}
	\item $\forall \alpha \in I(\KP_n), \ \modu{a_\alpha} < d_\alpha \text{ or } a_\alpha = 1$, $a_{\rho_0^-} \neq 1$, $a_{\sigma_0^+} = 1$ or $a_{\sigma_0^-} = 1$ but $a_{\sigma_0^+} \neq a_{\sigma_0^-}$, $\modu{a_{\sigma_2^+}}\!<\!1$ and $\modu{a_{\sigma_2^-}}\!<\!1$, i.e. $\mu = h_{\Gamma, l}$ with $q = 2$ and $l$ equals $0$ if $a_{\sigma_0^+} = 1$, or $1$ if $a_{\sigma_0^-} = 1$,
	\item $\forall \alpha \in I(\KP_n),\modu{a_\alpha} < d_\alpha \text{ or } a_\alpha = 1$, $a_{\rho_0^-} \neq 1$, $\modu{a_{\sigma_0^+}} < 1$ and $\modu{a_{\sigma_0^-}} < 1$ but condition \eqref{Hcv0} does not hold, i.e. $\mu = h_{\Gamma, l, \tau}$ with $q = 1$ (and $p = n$), $l = 0$ and $\tau_i = 1$ for all $i \in \Zn$,
	\item $\forall \alpha \in I(\KP_n), \ \modu{a_\alpha} < d_\alpha \text{ or } a_\alpha = 1$, $a_{\rho_0^-} \neq 1$, $a_{\sigma_0^+} = a_{\sigma_2^+} = 1$ or $a_{\sigma_0^-} = a_{\sigma_2^-} = 1$ but $a_{\sigma_0^+} \neq a_{\sigma_0^-}$, i.e. $\mu = h_{\Gamma, l, \tau}$ with $q = 2$ (and $p = \frac{n}{2}$),
	\item otherwise the random walk diverges.
\end{itemize}
\label{BigTh}
\end{theorem}

\begin{remark}
By Proposition \ref{IdemCent}, the limit states are central for the convolution product. One can check that an idempotent state arises as the limit of such a random walk if and only if it is central.
\end{remark}

\begin{proof}
We only need to look at the case when the limit $\mu$ exists but is not the Haar state or some $h_\Gamma$. We will again use the characterizations \eqref{CarhGl} and \eqref{CarhGlT} of the idempotent states $h_{\Gamma, l}$ and $h_{\Gamma, l, \tau}$.

The limit of the $\widehat{\Fo(a)^{\star k}}(\Xuv{u}{v})$'s is null for all $v$ between $1$ and $\lfloor \frac{n-1}{2} \rfloor$, which is equivalent to  $q$ equals $1$ or $2$.

Since $q$ divides $n$, if $n$ is odd, the only possibility is $q = 1$. Hence,  $\mu$ is the state $h_{\{0\}\times \Zn, 0, (1)_{i = 0}^{n-1}}$.

If $n$ is even, $\Xuv{0}{\frac{n}{2}}$ exists, and
\[\widehat{\Fo(a)^{\star k}}(\Xuv{0}{\frac{n}{2}}) = \frac{1}{2}\begin{pmatrix}a_{\sigma_{0}^+}^k + a_{\sigma_{0}^-}^k & a_{\sigma_{0}^+}^k - a_{\sigma_{0}^-}^k \\ a_{\sigma_{0}^+}^k - a_{\sigma_{0}^-}^k & a_{\sigma_{0}^+}^k + a_{\sigma_{0}^-}^k\end{pmatrix}\]
tends to $\begin{pmatrix}0&0\\0&0\end{pmatrix}$, $\frac{1}{2}\begin{pmatrix}1&1\\1&1\end{pmatrix}$ or $\frac{1}{2}\begin{pmatrix}1&-1\\-1&1\end{pmatrix}$.

Let us note that the limit is null if and only if $a_{\sigma_0^+}$ and $a_{\sigma_0^-}$ have both modulus less than $1$. Since, among the possible limit states, the equality $\widehat{\mu}(\Xuv{0}{\frac{n}{2}}) = \begin{pmatrix}0&0\\0&0\end{pmatrix}$ characterizes $h_{\{0\}\times \Zn, 0, (1)_{i = 0}^{n-1}}$, we get the fourth case.

Assume now that the previous limit is not the null matrix. The case $u = n - 2$ and $v = \frac{n}{2}$ will determine the limit, since
\[\widehat{h_{\Gamma, l}}(\Xuv{n-2}{\frac{n}{2}}) = \begin{pmatrix}0&0\\0&0\end{pmatrix} \ \ \text{ whereas }\ \ \widehat{h_{\Gamma, l, \tau}}(\Xuv{n-2}{\frac{n}{2}}) = \widehat{h_{\Gamma, l, \tau}}(\Xuv{0}{\frac{n}{2}})\text{ .}\]

We have
\[\widehat{\Fo(a)^{\star k}}(\Xuv{n-2}{\frac{n}{2}}) = \frac{1}{2}\begin{pmatrix}a_{\sigma_{2}^+}^k + a_{\sigma_{2}^-}^k & a_{\sigma_{2}^+}^k - a_{\sigma_{2}^-}^k \\ a_{\sigma_{2}^+}^k - a_{\sigma_{2}^-}^k & a_{\sigma_{2}^+}^k + a_{\sigma_{2}^-}^k\end{pmatrix}\]
whose limit is null, if and only if $\modu{a_{\sigma_2^+}}\!<\!1$ and $\modu{a_{\sigma_2^-}}\!<\!1$. This concludes the proof of the theorem. \qed
\end{proof}

\begin{remark}
If we look at the characterizations \eqref{CarhGl} and \eqref{CarhGlT}, we have also the following equivalences:
\begin{multline*}
\forall l \in \Zn, a_{\rho_l^+}=1 \text{ and } \modu{a_\alpha} < d_\alpha \text{ for all the others } \alpha \in I(\KP_n)\\
\Longleftrightarrow  \ \ \mu = h_{\{0\}\times\Zn, 1, (1)_{i = 0}^{n-1}}
\end{multline*}
\begin{multline*}
a_{\rho_0^+} = 1, a_{\sigma_0^+} = 1 \text{ and } \modu{a_\alpha} < d_\alpha \text{ for all the others } \alpha \in I(\KP_n)\\
\Longleftrightarrow \ \ \mu = h_{\Zn \times 2\Zn, 0}\hspace{0.9cm}
\end{multline*}
\begin{multline*}
a_{\rho_0^+} = 1, a_{\sigma_0^-} = 1 \text{ and } \modu{a_\alpha} < d_\alpha \text{ for all the others } \alpha \in I(\KP_n)\\
\Longleftrightarrow \ \ \mu = h_{\Zn \times 2\Zn, 1}\hspace{0.9cm}
\end{multline*}
\begin{multline*}
\forall i \in \mathbb{Z}_{\frac{n}{2}}, \ a_{\rho_{2i}^+} = 1, \forall i \in \mathbb{Z}_{\frac{n}{2}}, \ a_{\sigma_{2i}^+} = 1 \text{ or } \forall i \in \mathbb{Z}_{\frac{n}{2}}, \ a_{\sigma_{2i}^-} = 1 \text{ but } a_{\sigma_0^+} \neq a_{\sigma_0^-},\\
\text{ and } \modu{a_\alpha} < d_\alpha \text{ for all the others } \alpha \in I(\KP_n)\\
\Longleftrightarrow\ \ \mu = h_{\frac{n}{2}\Zn \times 2\Zn, l, \tau}\hspace{0.5cm}
\end{multline*}
\end{remark}

\begin{remark}
Thanks to Remark \ref{rkMod} we get that the random walk diverges if and only if there exists $\alpha$ in $I(\KP_n)$ such that $\modu{a_\alpha} = d_\alpha$ but $a_\alpha \neq d_\alpha$. In this case, the random walk can be cyclic or not.
\end{remark}

\begin{example}
Let us look again at the state $\phi_p$ defined in Example \ref{ex1}. By Theorem \ref{CondCv}, the random walk diverges. Moreover, we have
\[\phi_p^{\star k} = \sum\limits_{l = 0}^{q - 1} \eta^{klp} \Fo(\rho_{lp}^+)\text{ ,}\]
for $q > 1$ such that $pq = n$. Thus the random walk is cyclic, with period $q$.

Let us note that $\rho_0^+ - \rho_0^- = 2 \sum\limits_{m = 1}^n E_{m,m}$ clearly satisfies the conditions of Lemma 6.4 in \cite{idem}. The associated random walk is also cyclic, with period $2$, for all $n \geq 2$ (even for odd $n$). The state is the Haar state when $k$ is even and $2 \sum\limits_{m = 1}^n \Fo(E_{m,m})$ for $k$ odd.

Thus, we can define the random walk
\[\psi_p = \Fo\left( (\rho_0^+ - \rho_0^-) \sum_{l = 0}^{q-1} \eta^{lp}\rho_{lp}^+\right) = \sum_{l = 0}^{q-1} \left(\eta^{lp}\Fo(\rho_{lp}^+) - \eta^{lp} \Fo(\rho_{lp}^-)\right)\]
which is also cyclic, with period $\lcm(2, q)$. In particular, it has period $2q$ when $n$ is odd.
\end{example}

\begin{example}
Let us note that $\rho_0^+ - \rho_0^- + \frac{1}{2} \sum\limits_{l = 1}^{n-1} (\rho_l^+ - \rho_l^-)$ defines a state. One can check that the associated random walk is not cyclic and admits no limit. 
\end{example}

\subsection{Random walks in the dual of $\KP_n$}

We can define the dual group of a finite quantum group $\mathbb{G} = (C(\mathbb{G}), \Delta)$. This dual, denoted $\hat{\mathbb{G}} = (C(\hat{\mathbb{G}}), \hat{\cop})$, is again a finite quantum group. The algebra  $C(\hat{\mathbb{G}})$ is the set of all linear forms defined on $C(\mathbb{G})$. It is also isomorphic to the direct sum of the non-equivalent unitary irreducible corepresentations of $C(\mathbb{G})$, thanks to the Fourier transform. All the structures are defined by duality.

Let us denote the dual of $\KP_n$ by $\widehat{\KP_n} = (C(\widehat{\KP_n}), \hat{\cop})$. Its algebra admits the dual basis $\{e^{(i,j)}\}_{i,j \in \Zn} \cup \{E^{i,j}\}_{1 \leq i,j \leq n}$. Then its irreducible characters are all the $e^{(i,j)}$'s and $\chi(\hat{\X}) = \sum\limits_{i = 1}^n E^{ii}$. The Haar state, dual of the counit of $\KP_n$, is given by
\[\int_{\widehat{\KP_n}}\!\!\left(\sum_{i,j \in \Zn} x_{(i,j)}e^{(i,j)} + \sum_{1 \leq i,j \leq n} X_{i,j} E^{ij}\right) = x_{(0,0)} \text{ .}\]

The Fourier transform on $C(\widehat{\KP_n})$ is denoted $\Fo^{-1}$, since it allows us to go back in the Sekine finite quantum group $\KP_n$, identifying $\Fo^{-1}(e^{(i,j)})$ with $e_{(-i, -j)}$ and $\Fo^{-1}(E^{ij})$ with $\frac{1}{n} E_{ij}$.

Thus the Fourier transform of element $a = \hspace{-8pt}\sum\limits_{\alpha \in I(\widehat{\KP_n})}\hspace{-8pt}a_\alpha \chi_\alpha = \sum\limits_{i,j \in \Zn}\hspace{-3pt}a_{(i,j)}e^{(i,j)} + g_{\hat{\X}} \sum\limits_{i = 1}^n E^{ii}$ defines a state on $C(\widehat{\KP_n})$ if and only if $\Fo^{-1}(a)$ is a positive element in $C(\KP_n)$ and $\int_{\widehat{\KP_n}}\!\!a = 1$. This is equivalent to the fact that $a_\alpha$ is real and non negative for any $\alpha \in I(\widehat{\KP_n})$ and $a_{(0,0)} = a_{e^{(0,0)}} = 1$.

The same type of computation as above gives us the following result:
\begin{proposition}
The random walk associated to such an element $a$ of $C(\widehat{\KP_n})$ converges to the Haar state $\int_{\widehat{\KP_n}}$ if and only if
\[\forall \alpha \in I(\widehat{\KP_n})\setminus \{e^{(0,0)}\}, \ a_\alpha < d_\alpha \text{ .}\]
If this condition is not satisfied, but $a_\alpha \leq d_\alpha$ for any irreducible representation $\alpha$, then the random walk converges to some central idempotent state, i.e.\ an element of $C(\KP_n)$ of the form $\sum\limits_{i,j \in \Zn}\hspace{-2pt}\varepsilon_{(i,j)} e_{(i,j)} + \varepsilon_{\hat{\X}} \sum\limits_{i = 1}^n E_{ii}$, with $\varepsilon_\alpha = 0$ or $1$, for any $\alpha$ in $I(\widehat{\KP_n})$ and $\varepsilon_{(0,0)} = 1$. Moreover, every idempotent element in the center of $C(\KP_n)$ can be obtained this way.

Otherwise, there exists $\alpha$ such that $a_\alpha > d_\alpha$ and the random walk diverges.
\end{proposition}

\begin{proof}
Let us note that, for any $\beta \in I(\widehat{\KP_n})$
\[\widehat{\Fo^{-1}(a)}(\beta) = \begin{cases} a_{(-i, -j)} &\text{ if } \beta = e^{(i,j)}\\
\frac{a_{\hat{\X}}}{n} I_n &\text{ if } \beta = \hat{\X} \end{cases} \text{ .}\]
Thus, the random walk converges if and only if $a_\alpha < d_\alpha$ for all non trivial $\alpha \in I(\widehat{\KP_n})$.

Then the distance between the random walk and the Haar state can bounded from above and below in the following way:
\[\frac{1}{2} \left(\frac{a_\alpha}{d_\alpha}\right)^k \leq \norm{\Fo^{-1}(a)^{k} - \int_{\widehat{\KP_n}}}_{QTV} \leq \frac{1}{2} \sqrt{\sum_{\substack{\beta \in I(\widehat{\KP_n})\\ \beta \neq e^{(0,0)}}} \frac{a_\beta^{2k}}{d_\beta^{2(k-1)}}}\]
for any non trivial irreducible representation $\alpha$. This gives the necessary and sufficient condition for the random walk to converge to the Haar state.

Since $\widehat{\KP_n}$ is of Kac type, Proposition \ref{IdemCent} holds also in this case. Note that states on $C(\widehat{\KP_n})$ correspond to positive elements in $C(\KP_n)$ with coefficient $1$ in front of $e_{(0,0)}$, thus of the form $\sum\limits_{i, j \in \Zn} x_{(i,j)}e_{(i,j)} + \sum\limits_{i,j = 1}^n X_{i,j}E_{i,j}$ where the $x_{(i,j)}$'s are non negative and $X = (X_{i,j})_{1 \leq i,j \leq n}$ is positive semi-definite. The state is central if and only if the matrix $X$ is central in $M_n(\CC)$, i.e.\ is a scalar multiple of the identity matrix $I_n$. Then the state is  central and idempotent if and only if it is of the announced form.\qed
\end{proof}

\paragraph*{Acknowledgements :}
The author would like to thank Haonan Zhang, for providing her with his preprint \cite{Haonan}. This work was supported by the French "Investissements d'Avenir" program, project ISITE-BFC (contract ANR-15-IDEX-03)

\end{document}